\iffalse
\documentclass{mcom-l}
\usepackage{amssymb}
\usepackage{booktabs}
\usepackage{comment}

\usepackage[ruled,vlined]{algorithm2e}
\usepackage{comment}
\newtheorem{thm}{Theorem}
\newtheorem{lem}{Lemma}
\newtheorem{cor}{Corollary}
\newcommand{\F}{\mathbb{F}}
\newcommand{\Fq}{\F_q}
\newcommand{\Fqs}{\F_q^*}
\newcommand{\Fqn}{\F_{q^n}}

\newcommand{\Fqt}{\F_{q^3}}
\newcommand{\Fqts}{\F_{q^3}^*}
\newcommand{\Fqf}{\F_{q^4}}
\newcommand{\Fqfs}{\F_{q^4}^*}
\newcommand{\Ln}{\mathcal{L}_{n}}
\newcommand{\Tn}{\mathcal{T}_{n}}
\newcommand{\Lt}{\mathcal{L}_{3}}
\newcommand{\Lf}{\mathcal{L}_{4}}
\newcommand{\Tf}{\mathcal{T}_{4}}

\newcommand{\Magma}{\textsc{Magma}}
\newcommand{\MagmaV}{\Magma~V2.23}
\newcommand{\regsym}{\textsuperscript{\tiny\textregistered}}

\begin{document}

\title[Existence results for primitive elements]{Existence results for primitive elements in  cubic and quartic extensions of a finite field}

\author{Geoff Bailey}
\address{Computational Algebra Group, School of Mathematics and Statistics, University of Sydney, Australia}
\email{geoff.bailey@sydney.edu.au}
\curraddr{}
\thanks{}

\author{Stephen D. Cohen}
\address{School of Mathematics and Statistics, University of Glasgow, Scotland}
\email{stephen.cohen@glasgow.ac.uk}
\curraddr{}
\thanks{}

\author{Nicole Sutherland}
\address{Computational Algebra Group, School of Mathematics and Statistics, University of Sydney, Australia}
\email{nicole.sutherland@sydney.edu.au}
\curraddr{}
\thanks{}

\author{Tim Trudgian}
\address{School of Physical, Environmental and Mathematical Sciences, UNSW Canberra
at the Australian Defence Force Academy, Campbell, ACT 2610, Australia}
\email{t.trudgian@adfa.edu.au}
\curraddr{}
\thanks{Supported by Australian Research Council Future Fellowship FT160100094.}



\dedicatory{}

\begin{abstract}
  \noindent With $\Fq$ the finite field of $q$ elements, we investigate the following question.
If $\gamma$ generates $\Fqn$ over $\Fq$ and $\beta$ is a non-zero element of $\Fqn$, is there always
an $a \in \Fq$ such that $\beta(\gamma + a)$ is a primitive element? We resolve this case when $n=3$, thereby proving a conjecture by Cohen. We also improve substantially on what is known when $n=4$.
\end{abstract}

\maketitle

\else
\documentclass[11pt]{article}
\usepackage{a4wide}

\usepackage{booktabs}
\usepackage[ruled,vlined]{algorithm2e}

\usepackage{amsthm}
\usepackage{amsmath}
\usepackage{amssymb}
\usepackage{comment}
\newtheorem{thm}{Theorem}
\newtheorem{lem}{Lemma}
\newtheorem{cor}{Corollary}
\newtheorem{conjecture}{Conjecture}
\newcommand{\F}{\mathbb{F}}
\newcommand{\Fq}{\F_q}
\newcommand{\Fqn}{\F_{q^n}}
\newcommand{\Fqs}{\F_q^*}
\newcommand{\Fqt}{\F_{q^3}}
\newcommand{\Fqts}{\F_{q^3}^*}
\newcommand{\Fqf}{\F_{q^4}}
\newcommand{\Fqfs}{\F_{q^4}^*}
\newcommand{\Ln}{\mathcal{L}_{n}}
\newcommand{\Lt}{\mathcal{L}_{3}}
\newcommand{\Tn}{\mathcal{T}_{n}}
\newcommand{\Lf}{\mathcal{L}_{4}}
\newcommand{\Tf}{\mathcal{T}_{4}}

\newcommand{\Magma}{\textsc{Magma}}
\newcommand{\MagmaV}{\Magma~V2.23}
\newcommand{\regsym}{\textsuperscript{\tiny\textregistered}}

\title{Existence results for primitive elements in  cubic and quartic extensions of a finite field}

\author{
  Geoff Bailey\\
  Computational Algebra Group, \\
  School of Mathematics and Statistics, \\
  University of Sydney, Australia \\
  geoff.bailey@sydney.edu.au
 \and
  Stephen D. Cohen \\
  School of Mathematics and Statistics, \\
  University of Glasgow, Scotland \\
  stephen.cohen@glasgow.ac.uk
  \and
  Nicole Sutherland \\
  Computational Algebra Group, \\
  School of Mathematics and Statistics, \\
  University of Sydney, Australia \\
  nicole.sutherland@sydney.edu.au
\and
  Tim Trudgian\footnote{Supported by Australian Research Council Future Fellowship FT160100094.} \\
School of Physical, Environmental and Mathematical Sciences\\ The University of New South Wales Canberra, Australia \\
  t.trudgian@adfa.edu.au
}
\date{}

\begin{document}

\maketitle

\begin{abstract}
  \noindent With $\Fq$ the finite field of $q$ elements, we investigate the following question.
If $\gamma$ generates $\Fqn$ over $\Fq$ and $\beta$ is a non-zero element of $\Fqn$, is there always
an $a \in \Fq$ such that $\beta(\gamma + a)$ is a primitive element? We resolve this case when $n=3$, thereby proving a conjecture by Cohen. We also improve substantially on what is known when $n=4$.

\end{abstract}


\fi

\section{Introduction}
Let $q$ be a prime power and let $\Fq$ be the finite field of order $q$.
Suppose that $\gamma$ generates $\Fqn$ (over $\Fq$, as throughout);
thus $\Fqn = \Fq(\gamma)$.
Davenport \cite{Davenport1962} showed that whenever $q$ is a sufficiently
large prime there exists an $a\in \Fq$ such that $\gamma + a$ is a
primitive element of $\Fqn$.
This result was generalised for $q$ a prime power by Carlitz \cite{Carlitz}.

Consider the following problem: If $\gamma_1$ and $\gamma_2$ are non-zero
members of $\Fqn$ such that $\gamma_2/\gamma_1$ generates $\Fqn$,
is there always an $a \in \Fq$ such that $a\gamma_1 + \gamma_2$ is primitive?
Equivalently,
if $\gamma$ generates $\Fqn$ and $\beta \in \Fqn^*$, is there always
an $a \in \Fq$ such that $\beta(\gamma + a)$ is primitive?

Define $\Ln$ to be the set of all $q$ for which such an $a$ always exists
for any $\gamma_1$ and $\gamma_2$
(or $\beta$ and $\gamma$ in the alternative formulation)
satisfying the conditions.
The \textit{line problem} for degree $n$ extensions is to determine
which prime powers $q$ are in $\Ln$.

For quadratic extensions $\F_{q^2}$ Cohen \cite{Cohen87} proved that
there is always such a representation (i.e., that all prime powers $q$
are in $\mathcal{L}_{2}$).
In \cite[Thm 5.1]{CohenCubic1} he considered cubic fields and proved
\begin{thm}[Cohen]\label{tim:co}
Let $q\notin \{3,4,5,7,9,11,13,31,37\}$ be a prime power.
Unless $q$ is one of an explicit set of $149$ possible exceptions
(the largest of which is $q=9811$)
then $q\in\Lt$.
\end{thm}

Theorem 1.3 in \cite{CohenCubic2} establishes that there are at most 149
exceptional values of $q$, and these are listed in
\cite[Thm 6.4]{CohenCubic2}.%
\footnote{%
We remark that checking the derivation of \cite[Thm 6.4]{CohenCubic2}
confirms that there are indeed 149 exceptions;
however, $q = 2221$ is an exception not listed there, while $q=4096$, which
is listed, can be removed by taking $t=1$ in \cite[Prop 6.1]{CohenCubic2}.}
For completeness, a full list of the possible exceptions (modified as explained in Section~\ref{West}) is given in Corollary \ref{co} below.

The principal goal of this paper is to resolve the line problem for
cubic extensions completely by proving

\begin{thm}\label{L3complete}
Let $q$ be a prime power.
Then $q\in\Lt$ iff $q\notin\{3,4,5,7,9,11,13,31,37\}$.
\end{thm}

The outline of this paper is as follows. In Section \ref{West} we outline an improvement of the modified prime sieve, as used by Cohen \cite{CohenCubic2}. This, and the results in Section \ref{sec3}, allow us to reduce the list of possible exceptions of Theorem \ref{tim:co}. In Section \ref{Comp} we outline the computational complexity in verifying that an element satisfies Theorem \ref{tim:co}, and in Section \ref{Nike} we present the results of our computations. These allow us to prove Theorem \ref{L3complete}. Finally, in Section \ref{square} we give an improvement on what is known for quartic extensions.

\section{A refinement of the modified prime sieve}\label{West}


Consider \cite[Prop 4.1]{CohenCubic1} and its generalisation to extensions
of degree $n$ in \cite[Prop 6.3]{CohenCubic2}.
This modifies the sieving argument given in \cite[Prop 3.3]{CohenCubic1} by
treating specially one of the sieving primes $l$
(in practice the largest prime divisor).
We can extend this by treating specially $r$ primes $l_{1}, \ldots, l_{r}$
(in practice the largest $r$ prime divisors).

Throughout, for any positive integer $m$, define $\theta(m)= \phi(m)/m$, where $\phi$ denotes Euler's function. Also, by the {\em radical}
of  $m$ we shall mean the product of the {\em distinct} primes of $m$. 

The function $N$ is as defined in \cite{CohenCubic1} and \cite{CohenCubic2}.  Briefly, assume $\alpha, \beta \in \mathbb{F}_{q^n}^*$
are given (with $\beta$ being a generator),  as described in the definition of ${\mathcal{L}}_n$.  Then, for any divisor $e$ of $q^n-1$, $N(e)$ is 
the number of $a \in \mathbb{F}_q$ such that  $\beta(\alpha +a)$ is $e$-free, where an $e$-free element of $\mathbb{F}_{q^n}^*$ is one that can only be written
as $\gamma^d, \gamma \in \mathbb{F}_{q^n}$,  for a divisor $d$ of $e$ if $d=1$. 
 
 Our first lemma extends \cite[Prop 6.3]{CohenCubic2} and its proof follows a similar pattern.
\begin{lem}\label{improvedsieve}
Let $q$ be a prime power:  write the radical of $q^{n} -1$ as
$kp_{1}\cdots p_{s} l_{1}\cdots l_{r}$, where $k$ has $t$ distinct
prime divisors and $p_{1}, \ldots, p_{s}, l_{1}, \ldots, l_{r}$
are distinct prime numbers.
Set
$m= \theta(k)$,
$\delta = 1 - \sum_{i=1}^{s} \frac{1}{p_{i}}$ and
$\varepsilon = \sum_{j=1}^r \frac{1}{l_j}$.
If $\delta m > \varepsilon$ and if
\begin{equation*}\label{eq:cry}
q> (n-1)^{2} \left\{ \frac{2^{t}m(s-1+ 2\delta) -m\delta + r - \varepsilon}{m \delta - \varepsilon}\right\}^{2},
\end{equation*}
then $q\in \Ln$.
(Here, by convention, if $s=0$ then $\delta=1$,
and if $r=0$ then $\varepsilon=0$.)
\end{lem}

\begin{proof}
Apply \cite[Lemma 4.1]{CohenCubic2} $r$ times, showing that
\begin{equation}\label{eq:1}
\begin{split}
N(q^{n} - 1)
&
\geq N(k p_{1} \cdots p_{s}) + \sum_{j=1}^rN(l_{j})\  - rN(1)\\
&
= N(k p_{1} \cdots p_{s}) + \sum_{j=1}^r\left(N(l_{j})- \left(1-\frac{1}{l_j}\right)N(1)\right)-\varepsilon N(1).
\end{split}
\end{equation}
From \cite[Lemma 4.2]{CohenCubic2} (just a further $s-1$ applications
of \cite[Lemma 4.1]{CohenCubic2})
\begin{equation}\label{eq:2}
N(k p_{1} \cdots p_{s})  \geq \delta N(k)+ \sum_{i=1}^s\left(N(kp_i)-\left(1-\frac{1}{p_i}\right)N(k)\right).
\end{equation}

Of course, $N(1)=q$.  Also, from \cite[Cor 2.3]{CohenCubic2},
\begin{equation}\label{eq:Nk}
N(k) \geq \theta(k)(q -(n-1)(2^t-1) \sqrt{q});
\end{equation}
for $i=1, \ldots, s$,
\begin{equation} \label{eq:Nkpi}
\left|N(kp_i)-\left(1-\frac{1}{p_i}\right)N(k)\right|\leq \left(1-\frac{1}{p_i}\right)\theta(k)(n-1)2^t \sqrt{q};
\end{equation}
and, for $j =1, \ldots,r$,
\begin{equation}\label{eq:Nlj}
\left|N(l_j)-\left(1-\frac{1}{l_j}\right)N(1)\right|\leq \left(1- \frac{1}{l_j} \right)(n-1)\sqrt{q}.
\end{equation}

Applying  (\ref{eq:Nk}), (\ref{eq:Nkpi}) and (\ref{eq:Nlj}) to  (\ref{eq:1}) and  (\ref{eq:2}) and using the definitions of $m$, $\delta$and $\varepsilon$, we obtain
$$N(q^n-1)\geq (m\delta -\varepsilon)q -(n-1)\{2^t(s-1+2\delta)m-m\delta +(r-\varepsilon)\}\sqrt{q}.   $$

The criterion of the lemma follows.
\end{proof}

The possible gain in using Lemma \ref{improvedsieve} in lieu of
\cite[Prop 6.3]{CohenCubic2} stems from the reduction in~$s$.
Provided that the primes $l_{1}, \ldots,  l_{r}$ are sufficiently large,
the reduction in $s$ may offset the loss of a slightly smaller value of
$\delta$.

Restricting to $n=3$, we can use Lemma \ref{improvedsieve} to eliminate
three values of $q$ from Cohen's list $S$ in
\cite[Thm 4.2]{CohenCubic1}.
Setting $r = 0$ and $t = 1$ suffices\footnote{When $r = 0$, Lemma \ref{improvedsieve} is slightly better \cite[Prop 3.3]{CohenCubic1} --- just better to rule
out this one case.} to eliminate $q = 809$,
while choosing $r = 2$ and $t = 2$ allows us to rule out
$q = 1951$ and $q = 5791$.
This proves
\begin{cor}\label{co}
Let $q\notin \{3,4,5,7,9,11,13,31,37\}$.
Then $q\in \Lt$ except possibly for $146$ values of $q$.
These are
\begin{equation}\label{scone}
\begin{split}
\{&101, 103, 107, 109, 113, 121, 125, 127, 131, 137, 139, 149, 151, 157,
163, 169, 179, 181,\\ &191,
193, 197, 199, 211, 223, 229, 233, 239, 241,
243, 251, 256, 263, 269, 271, 277, 281,\\ &283, 289,
 307, 311, 313, 331,
337, 343, 347, 349, 359, 361, 367, 373, 379, 397, 419, 421, \\ &431, 439,
443, 457, 461, 463, 491, 499, 521, 523, 529, 541, 547, 571, 601, 607,
613, 619,\\ & 625, 631, 661, 691,
 709, 729, 739, 751, 757, 811, 821,
823, 841, 859, 877, 919, 961, 967,\\ &991, 997, 1021, 1033,1051, 1069,
1087, 1123, 1129, 1171, 1201, 1231, 1291, 1303, 1321, \\ &1327, 1369,
1381, 1429, 1451, 1453, 1471, 1531, 1597, 1621, 1681, 1741, 1831,
1849, \\ & 1871, 1873, 2011, 2209, 2221,2311, 2347, 2401, 2473, 2531,
2551, 2557, 2671,\\ & 2731, 2851, 2857, 2971, 3481, 3571, 3691, 3721,
4111, 4561, 4951, 5821, 6091, 9811\}.
\end{split}
\end{equation}
\end{cor}
While Lemma \ref{improvedsieve} only allows us to remove three values of $q$ from the list of exceptions in \cite{CohenCubic2}, it will play an essential role in our work on quartic extensions in \S \ref{square}. In the next section we make more substantial progress on the list of possible exceptions in Corollary \ref{co}.

\section{An improvement of Katz' lemma for cubic extensions}\label{sec3}

Let $\chi$ be a multiplicative character of $\Fqn$, whence the order
of $\chi$ is a divisor of $q^n-1$.
For any $\gamma \in \Fqn$, define
$S_\gamma(\chi)=\sum_{a \in \Fq} \chi(\gamma+a)$.
A key tool for attacking existence questions for primitive elements in
extensions has been a deep result of Katz \cite{Katz}.

\begin{lem}[Katz]\label{katz}
Suppose that $\gamma$ generates $\Fqn$ over $\Fq$ and $\chi$ is a
non-principal character of $\Fqn$ (i.e., has order exceeding 1).
Then $|S_\gamma(\chi)| \leq (n-1) \sqrt{q}$.
\end{lem}

When $n=2$, Lemma \ref{katz} shows that $|S_\gamma(\chi)| \leq \sqrt{q}$.
Prior to the publication of \cite{Katz}, Cohen \cite{Cohen87} had proved
this result elementarily with an improved bound
when the non-principal character $\chi$ had order dividing $q+1$.
When $n=3$, Lemma \ref{katz} shows that $|S_{\gamma}(\chi)| \leq 2\sqrt{q}$.
Whereas we cannot offer an alternative proof of this result in general,
we can establish an improvement with an elementary proof in the case
in which the non-principal character $\chi$ has order dividing $q^2+q+1$.
This might be viewed as an analogue of the improvement in the quadratic case.

\begin{lem} \label{better}
Let $\beta, \gamma$ be non-zero elements of $\Fqt$ such that
$\gamma$ generates $\Fqt$.
Also let $\chi$ be a non-principal character of $\Fqt$
whose order divides $q^2+q+1$.
Then
$$ \left|\sum_{a \in \Fq}\chi(\beta(\gamma+a))\right| \leq \sqrt{q}+1.$$
\end{lem}

\begin{proof}
The significance of the restriction on the order of $\chi$
is that $\chi(c) = 1$ for all $c\in\Fqs$,
since such $c$ are $(q^2 + q + 1)$-th powers in $\Fqt$.
Furthermore, observe that the sum in question is
$\chi(\beta)S_{\gamma}(\chi)$ which has the same absolute value as
$S_\gamma(\chi)$.  Hence it suffices to show that

\begin{equation}\label{Seq}
|S_{\gamma}(\chi)| \leq \sqrt{q}+1.
\end{equation}
Abbreviate $S_{\gamma}(\chi)$ to $S$ and denotes its complex conjugate by $\bar{S}$. Then
\begin{equation}\label{S1}
|S|^2= S\bar{S} = \sum_{a,b \in \mathbb{F}_q} \chi\left( \frac{\gamma+a}{\gamma+b}\right)=\frac{1}{q-1}\sum_{\substack{a,b,c \in \mathbb{F}_q\\c \neq 0}} \chi\left( \frac{c(\gamma+a)}{\gamma+b}\right).
\end{equation}

Next we investigate the set
$\mathcal{T}=\left\{\frac{c(\gamma+a)}{\gamma+b}: \ a, b,c \in
\mathbb{F}_q, c \neq 0 \right\}$
appearing in (\ref{S1}) and compare
this to the set of
non-zero elements of $\Fqt$.
We claim
that the subset  $\mathcal{T}_0$ of  $\mathcal{T}$ comprising those
members for which $a\neq b$ is a set of $(q-1)^2q$
distinct elements, none of which is in $\mathbb{F}_q$.

To see this, suppose that
$\frac{c_1(\gamma+a_1)}{\gamma+b_1} = \frac{c_2(\gamma+a_2)}{\gamma+b_2}$;
then $c_1(\gamma+a_1)(\gamma+b_2)=c_2(\gamma+a_2)(\gamma+b_1)$.
Now, $\{\gamma^2,\gamma,1\}$ is a basis of $\Fqt$ over $\Fq$,
since $\gamma$ generates the extension.
It follows that
$c_1 = c_2$, $a_1 + b_2 = a_2 + b_1$, and $a_1b_2 = a_2b_1$,
whence $(a_2 - b_2)(b_1 - b_2) = 0$.
We have $a_2\neq b_2$ (by definition of $\mathcal{T}_0$), so
it follows that $b_1=b_2$ and $a_1=a_2$.
Thus elements of $\mathcal{T}_0$ can only be equal if they are identical,
and the claim is established.

The members of $\mathcal{T}\setminus\mathcal{T}_0$
comprise $\{c: a,b,c\in \Fq,a=b,c\neq0\}=\Fqs$,
each element $c \in \Fqs$ occurring with multiplicity $q$ in $\mathcal{T}$.
 Thus the cardinality of $\mathcal{T}$ as a subset of $\Fqts$ (discounting multiplicities) is $q(q-1)^2+(q-1)=(q-1)(q^2-q+1)$. Hence, the cardinality of $\mathcal{U}$,
 defined as the complement  of $\mathcal{T}$ in $\Fqts$,
  is $(q^3-1)-(q-1)(q^2-q+1)=2q(q-1)$.  Indeed, we can identify precisely the elements of $\mathcal{U}$ as follows.
 Suppose $\frac{c(\gamma+a)}{\gamma+b}=u(\gamma+v)$, where $a,b,c,u,v \in \mathbb{F}_q$ with $cu \neq 0$.  Then $c(\gamma+a)=u(\gamma+b)(\gamma+v)$.  Again because $\{\gamma^2,\gamma,1\}$
 is a basis,  this implies $c=0$, a contradiction.  It follows that $\mathcal{U}_1=\{u(\gamma+v), u,v \in \mathbb{F}_q, u \neq 0\} \subseteq \mathcal{U}$.
  Similarly, $\mathcal{U}_{-1}$, the set of reciprocals of members of $\mathcal{U}_1$ satisfies $\mathcal{U}_{-1} \subseteq \mathcal{U}$.  Moreover, $\mathcal{U}_1$ and $\mathcal{U}_{-1}$
  are disjoint sets each of cardinality $q(q-1)$.  From the cardinalities, we conclude that $\mathcal{U}= \mathcal{U}_1 \cup \mathcal{U}_{-1}$.

 The facts established in the previous paragraph applied to  (\ref{S1}) yield
 \begin{equation} \label{S3}
 |S|^2= \frac{1}{q-1}\bigg\{\sum_{\xi \in \Fqt ^*}\chi(\xi)
 + (q-1)\sum_{c\in \Fqs}\chi(c)
 - \sum_{\substack{u,v \in \Fq\\u\neq0}}\left(\chi(u(\gamma+v))+\chi\left(\frac{1}{u(\gamma+v)}\right)\right)\bigg\}.
 \end{equation}
The first sum in (\ref{S3}) is zero,
and $\chi(c) = 1$ for $c\in\Fqs$.
Accordingly,
 $$|S|^2=q-1-S- \bar{S} \leq q-1 +|S|+|\bar{S}|=q-1+2|S|.$$
Hence $(|S|-1)^2 \leq q$ and the inequality $(\ref{Seq})$ follows.
\end{proof}


Applying the better bounds of Lemma \ref{better} gives two useful
improvements to Lemma \ref{improvedsieve}.

\begin{lem}\label{LemC}
Take $n=3$ and adopt the same notation as in Lemma \ref{improvedsieve}.
Assume that $l_1, \ldots,l_r$ divide $q^2+q+1$.
Define
 $$\nu_1=\sum_{\substack{i=1\\ p_i\nmid (q^2+q+1)}}^s\frac{p_i-1}{p_i};
 \quad \nu_2=\sum_{\substack{i=1\\ p_i| (q^2+q+1)}}^s\frac{p_i-1}{p_i}.$$

For odd $q$ we may take $k = 2$ so that $t = 1$ and $m = \frac{1}{2}$.
If
\begin{equation}\label{eq:LemCk2}
q(m\delta - \varepsilon)
- \sqrt{q}(m(2\delta + 4\nu_1 + 3\nu_2) + r - \varepsilon)
- (m\nu_2 + r - \varepsilon)
> 0
\end{equation}
then $q \in \Lt$.

Alternatively, for $q\equiv 1 \pmod 6$ we may take $k = 6$ so that $t = 2$
and $m = \frac{1}{3}$.
If
\begin{equation}\label{eq:LemCk6}
q(m\delta - \varepsilon)
- \sqrt{q}(m(5\delta + 8\nu_1 + 6\nu_2) + r - \varepsilon)
- (m(\delta + 2\nu_2) + r - \varepsilon)
> 0
\end{equation}
then $q \in \Lt$.
\end{lem}

\begin{proof}
The proof uses the plan of Lemma \ref{improvedsieve} with appropriate
adjustments to the constants arising from the bounds of Lemma \ref{katz}
and Lemma \ref{better}.

For $k = 2$ no improvement applies over (\ref{eq:Nk}), so we have
$N(k) \geq m(q - 2\sqrt{q})$.
However, when considering $N(kp_i) - (1 - \frac{1}{p_i})N(k)$,
the underlying formula uses characters of order $p_i$ and $2p_i$.
This gives an improvement over (\ref{eq:Nkpi}) for the character of
order $p_i$ when $p_i \mid q^2 + q + 1$.
Moreover, $l_j \mid q^2 + q + 1$ so we always get an improvement
over (\ref{eq:Nlj}).

Thus for $k = 2$ we have
\begin{equation*}
\begin{split}
N(k)
&
\geq m(q - 2\sqrt{q})\\
\left|N(kp_i)-\left(1-\frac{1}{p_i}\right)N(k)\right|
&
\leq
\begin{cases}
& \makebox[11em][l]{$\left(1-\frac{1}{p_i}\right)m (4\sqrt{q})$}
\mbox{ if } p_i \nmid q^2+q+1\\
& \makebox[11em][l]{$\left(1-\frac{1}{p_i}\right)m (3\sqrt{q} + 1)$}
\mbox{ if } p_i \mid q^2+q+1\\
\end{cases}\\
\left|N(l_j)-\left(1-\frac{1}{l_j}\right)N(1)\right|
&
\leq \left(1 - \frac{1}{l_j}\right)(\sqrt{q} + 1).
\end{split}
\end{equation*}
Applying these revised bounds in the proof for Lemma \ref{improvedsieve}
gives (\ref{eq:LemCk2}).

For $k = 6$ we proceed similarly, noting that $3 \mid q^2 + q + 1$.
The character sum for $N(6)$ involves characters of order 1, 2, 3,
and 6; we can apply the improved bound for order 3, getting
$N(k) \geq m(q - 5\sqrt{q} - 1)$.
When considering $N(kp_i) - (1 - \frac{1}{p_i})N(k)$, the characters
involved have orders $p_i$, $2p_i$, $3p_i$, and $6p_i$; the improved
bounds apply for $p_i$ and $3p_i$ when $p_i \mid q^2 + q + 1$.
As before, we always get better bounds for the~$l_j$.

So for $k = 6$ we get
\begin{equation*}
\begin{split}
N(k)
&
\geq m(q - 5\sqrt{q} - 1)\\
\left|N(kp_i)-\left(1-\frac{1}{p_i}\right)N(k)\right|
&
\leq
\begin{cases}
& \makebox[11em][l]{$\left(1-\frac{1}{p_i}\right)m (8\sqrt{q})$}
\mbox{ if } p_i \nmid q^2+q+1\\
& \makebox[11em][l]{$\left(1-\frac{1}{p_i}\right)m (6\sqrt{q} + 2)$}
\mbox{ if } p_i \mid q^2+q+1\\
\end{cases}\\
\left|N(l_j)-\left(1-\frac{1}{l_j}\right)N(1)\right|
&
\leq \left(1 - \frac{1}{l_j}\right)(\sqrt{q} + 1).
\end{split}
\end{equation*}
Applying these revised bounds in the proof for Lemma \ref{improvedsieve}
gives (\ref{eq:LemCk6}).
\end{proof}

We now apply Lemma \ref{LemC} to the list of 146 possible exceptions
given by Corollary \ref{co}.
Using $k = 2$ we apply (\ref{eq:LemCk2}) for $r=0, 1, 2$
which eliminates all but 96 elements from our initial list.
For those remaining cases where $q \equiv 1 \pmod 6$ we then apply
(\ref{eq:LemCk6}) for $r = 0, 1, 2$;
this reduces the number of potential exceptions to 82, establishing

\begin{cor}\label{cor82}
Let $q\notin \{3,4,5,7,9,11,13,31,37\}$.
Then $q\in \Lt$ except possibly for $82$ values of $q$.
These are
\begin{equation}\label{Tie}
\begin{split}
\{
&
103, 107, 109, 113, 121, 125, 127, 131, 137, 139, 149, 151, 157, 163,
169, 181, 191, 193,\\
&
199, 211, 229, 239, 241, 256, 263, 271, 277, 281, 283, 289, 307, 311,
331, 337, 343, 349,\\
&
361, 367, 373, 379, 397, 421, 431, 457, 463, 499, 529, 541, 547, 571,
601, 625, 631, 661,\\
&
691, 751, 811, 823, 841, 877, 919, 961, 967, 991, 1171, 1231, 1303,
1321, 1327, 1369,\\
&
1381, 1597, 1831, 1849, 2011, 2311, 2671, 2731, 3571, 3721, 4111, 4951
\}.
\end{split}
\end{equation}
\end{cor}
While it does not seem possible to make any further theoretical advances by
modifying Lemma~\ref{LemC}, we note that the largest element in (\ref{Tie}) is considerably smaller than the largest element in (\ref{scone}). This reduction allows us to proceed with direct computation on the elements in (\ref{Tie}). The next sections give details of
computational arguments that eliminate the remaining exceptions, thereby proving Theorem~\ref{L3complete}.

\section{Computational complexity}\label{Comp}

Let $\beta$ and $\gamma$ be elements of $\Fqt$.
We call the pair $(\beta, \gamma)$ \emph{potentially bad} if
$\beta \neq 0$ and $\gamma$ generates $\Fqt$ over $\Fq$
(i.e., $\gamma \notin \Fq$).
Given a potentially bad pair $(\beta, \gamma)$, we call the pair \emph{good}
if there exists some $a \in \Fq$ such that $\beta(\gamma + a)$ is primitive;
otherwise we call it \emph{bad}.
Then $q\in\Lt$ iff all potentially bad pairs are good.

The number of potentially bad pairs is $(q^3 - 1)(q^3 - q)$, but we can
reduce the number that need checking through two observations.
For convenience in the following discussion,
fix $\omega$ to be a primitive element of $\Fqt$
and let $\tau = (q^3 - 1) / (q - 1) = q^2 + q + 1$.

Firstly, for any $\lambda \in \Fq$ the pair $(\beta, \gamma + \lambda)$
is good iff $(\beta, \gamma)$ is.
Thus we only need to check one value of $\gamma$ in each additive
coset with respect to $\Fq$.
More concretely, we can write
$\gamma = \gamma_2\omega^2 + \gamma_1\omega + \gamma_0$ with
$\gamma_i \in \Fq$, and the previous observation shows that we need
only consider pairs where $\gamma_0 = 0$.
This observation saves a factor of $q$, reducing the number of pairs
that need to be considered down to $(q^3 - 1)(q^2 - 1)$.

Secondly, for any $\lambda \in \Fqs$ the pair
$(\beta, \lambda\gamma)$ is good iff the pair $(\lambda\beta, \gamma)$
is good.
In the former case we check for badness by considering the values
$\beta(\lambda\gamma + a) = \lambda\beta\gamma + a\beta$
for all $a \in \Fq$,
while in the latter we consider the values
$\lambda\beta(\gamma + a) = \lambda\beta\gamma + \lambda a\beta$.
But $\lambda a$ also covers all values in $\Fq$, just in a different
order, so these sets are the same.
This observation allows us to check only one item in each multiplicative
coset with respect to $\Fqs$, saving a further factor of $q - 1$ and reducing
the number of pairs that need to be considered to
$(q^3 - 1)(q + 1)$.

There is a choice as to how to apply this multiplicative reduction.
If it is applied to $\beta$ then we have to choose a suitable set of
representatives;
a simple option is to let $\beta = \omega^k$ for $0 \leq k < \tau$,
since the elements of $\Fqs$ are precisely the powers of $\omega^\tau$.
This leads to considering the pairs
$(\omega^k, \gamma_2\omega^2 + \gamma_1\omega)$ for $0 \leq k < \tau$
and $\gamma_1, \gamma_2 \in \Fq$, not both zero.

Note that by our previous observation about elements of $\Fqs$ we can
write nonzero $\gamma_1$ and $\gamma_2$ as powers of $\omega^\tau$.
So an equivalent set of $\gamma$ to consider is the values
$\omega^{1 + k_1\tau}$, $\omega^{2 + k_2\tau}$, and
$\omega^{2 + k_2\tau} + \omega^{1 + k_1\tau}$
where $0 \leq k_1, k_2 < q - 1$.
Algorithm~\ref{alg1}
uses this alternative presentation; it turned out to be faster in
practice.

\begin{algorithm}[ht]

\DontPrintSemicolon
\AlgoDontDisplayBlockMarkers
\SetAlgoNoEnd
\SetAlgoNoLine
\SetKwProg{Proc}{Procedure}{}{}%
\SetKwFunction{checkgamma}{check\_gamma}%
\SetKwFunction{checkq}{check\_q}%
\Proc{\checkq{$q$}}{
    \textit{Construct $\Fq$, $\Fqt$, and $\omega$}\;
    $\tau \leftarrow q^2 + q + 1$\;
    \For{$0 \leq k < q - 1$}{
	\checkgamma{$\omega^{1 + k\tau}$}\;
	\checkgamma{$\omega^{2 + k\tau}$}\;
    }
    \For{$0 \leq k_1 < q - 1$}{
	\For{$0 \leq k_2 < q - 1$}{
	    \checkgamma{$\omega^{2 + k_2\tau} + \omega^{1 + k_1\tau}$}\;
	}
    }
}
\;
\Proc{\checkgamma{$\gamma$}}{
    \For{$0 \le k < \tau$}{
	$\beta \leftarrow \omega^k$\;
\SetKw{KwTo}{in}%
	\For{$a$ \KwTo $\Fq$}{
	    \If{$\beta(\gamma + a)$ is primitive}{
\SetKw{KwNext}{next}%
		\KwNext $k$\;
	    }
	}
	FAIL\;
    }
}
\caption{Check whether $q$ is good using reduced $(\beta, \gamma)$ pairs\label{alg1}}
\end{algorithm}

Alternatively, we could apply the multiplicative reduction to $\gamma$:
the pairs to be considered become
$(\beta, \omega^2 + \gamma_1\omega)$ and $(\beta, \omega)$ for
$\beta \in \Fqts, \gamma_1 \in \Fq$.
Additionally, let $R$ be the radical of $q^3 - 1$;
then $\omega^k$ is primitive iff $k$ is coprime to $R$
(equivalently, iff $\gcd(k, R) = 1$).
This property is unchanged by reduction modulo $R$; hence we need
only consider $\beta = \omega^k$ with $k < R$.

A reformulation of the problem allows us to do even better:
$\beta(\gamma + a) = \beta\gamma(1 + a/\gamma)$, and
as $\beta$ iterates through $\Fqts$ so does $\beta\gamma$.
So this is equivalent to considering the values $\beta'(1 + a/\gamma)$,
where $\beta' \in \Fqts$ and $\gamma$ is one of the values
$\omega^2 + \gamma_1\omega$ ($\gamma_1 \in \F_q$) or $\omega$.

This alternative version provides two benefits that lead to
practical time savings.
First, setting $a = 0$ in $\beta'(1 + a/\gamma)$ yields $\beta'$
regardless of the value of $\gamma$, so if $\beta'$ is primitive
then all associated pairs are automatically good.
It is thus only necessary to test non-primitive values of $\beta'$.

Second, a small simplification of the $\gamma$ values used in
this method is possible.
Calculating $1/(\omega + u)$ as a function of $u$,
we see that each $u \in \Fq$ gives rise to a different class representative
$\omega^2 + \gamma_1\omega$.%
\footnote{Explicitly, $\gamma_1 = f_2 - u$, where
$\omega^3 + f_2\omega^2 + f_1\omega + f_0 = 0$.}
For a given $\beta' \in \Fqts$ this allows us to use the slightly nicer
values $\beta'(1 + a/\omega)$ and $\beta'(1 + a(\omega + u))$, $u \in \Fq$.

This approach is shown in Algorithm~\ref{alg2}.
Although it has the same asymptotic complexity as Algorithm~\ref{alg1},
it usually iterates fewer times and is considerably faster in
practice.
For some values of $q$ for which both were tested, Algorithm~\ref{alg2} was more
than 400 times faster.

\begin{algorithm}[ht]
\DontPrintSemicolon
\AlgoDontDisplayBlockMarkers
\SetAlgoNoEnd
\SetAlgoNoLine
\SetKwProg{Proc}{Procedure}{}{}%
\SetKwFunction{checkbetagamma}{check\_beta\_inv\_gamma}%
\SetKwFunction{checkq}{check\_q}%
\SetKwFunction{rad}{rad}%
\SetKw{KwNext}{next}%
\SetKw{KwTo}{in}%
\Proc{\checkq{$q$}}{
    \textit{Construct $\Fq$, $\Fqt$, and $\omega$}\;
    $R \leftarrow$ \rad($q^3 - 1$)\;
    \For{$0 \leq k < R$}{
	$\beta \leftarrow \omega^k$\;
	\If{$\beta$ is primitive}{
	    \KwNext $k$\;
	}
	\checkbetagamma{$\beta, 1/\omega$}\;
	\For{$u$ \KwTo $\Fq$}{
	    \checkbetagamma{$\beta, \omega + u$}\;
	}
    }
}
\;
\Proc{\checkbetagamma{$\beta, \gamma^{-1}$}}{
    \For{$a$ \KwTo $\Fqs$}{
	\If{$\beta(1 + a\gamma^{-1})$ is primitive}{
	    \Return\;
	}
    }
    FAIL\;
}
\caption{Check whether $q$ is good using reduced $(\beta, \gamma^{-1})$ pairs\label{alg2}}
\end{algorithm}

\section{Computation}\label{Nike}
Initial computation was undertaken using \MagmaV~\cite{magma223}, with
early estimates indicating that some of the $q < 1000$ would take about
a year to complete.
An improvement was made by changing a \Magma{} setting to ensure that
the finite fields involved used the Zech logarithm representation
(which is more computationally efficient but requires more memory);
doing so reduced those estimates to less than eight months.

Implementing Algorithm~\ref{alg1} reduced these times to about three
months for
$q < 1000$ and implementing Algorithm~\ref{alg2}
further reduced these times to at most two weeks.
These computations were completed, so it has been checked by \MagmaV{}
that each $q < 1000$ in Corollary~\ref{cor82} is good.

In Table~\ref{qlt1000} we give the minimum, average, and maximum times
(\MagmaV, 2.6GHz Intel\regsym{} Xeon\regsym{} E5-2670) for checking $q$
listed in Corollary~\ref{cor82} in given ranges
using an implementation of Algorithm~\ref{alg2} in \Magma.
As can be seen from these timings, $q < 1000$ can be checked in less
than 15 days each.
In fact 62 of these 64 $q$ can be checked in less than 6 days each,
53 in less than a day each, and 29 in less than 1 hour each.

\begin{table}[ht]
\centering
\begin{tabular}{|l|r|r|r|r|r|r|}
\hline
$q$ range & $(100, 200)$ & $(200, 400)$ & $(400, 600)$ & $(600, 800)$ \\
\hline
Minimum & 11.3 s & 69 s & 881 s & 3.4 hrs \\
Average & 333.5 s & 2.03 hrs & 19.4 hrs & 1.5 days \\
Maximum & 722 s & 7.72 hrs & 2.4 days & 4.5 days \\
\hline
$q$ range & $(800, 1000)$ & $(1000, 2000)$  & & \\
\hline
Minimum & 2.5 hrs & 129.13 days & & \\
Average & 4.31 days &  & & \\
Maximum & 14.634 days &  & & \\
\hline
\end{tabular}
\smallskip
\caption{Timings for checking $q < 1000$, $q$ listed in Corollary~\ref{cor82}.}
\label{qlt1000}
\end{table}
The memory overhead of the Zech logarithm representation prohibits its
use for $q > 1000$ in general, mandating a switch to a more general implementation
of finite fields.
This impact is seen in the last column of Table~\ref{qlt1000}. It is clearly not practical to use this approach for larger $q$.

Instead, a highly specialised and optimised stand-alone program was
written to perform the computations.
This program first calculates a table of all reduced $(\gamma,a)$
pairs together with their logarithms (with respect to the primitive
element).
Then, for each $\gamma$, it loops through the values of $\beta$
and checks as many $a$ as necessary.

Primitivity testing can be done very easily using logarithms,
as previously mentioned.
Thus this stage does not need to construct any elements of the finite
field; instead, the loop is over the logarithm of $\beta$,
which is combined with the logarithms from the table.
Further refinements enable even the $\gcd$ to be eliminated, and
some heuristic (anti-)sorting reduces the number of $a$ that need to
be checked in practice.
Source code and a detailed explanation of the program may be found
at~\cite{BaileyL3}.

This program was used to test all prime powers $q < 5\,000$, using
24 threads on a
2.3GHz Intel\regsym{} Xeon\regsym{} E5-2699.
All $q < 2\,000$ had been checked after 12.3 hours,
and the remaining six values of $q > 2\,000$ in Corollary~\ref{cor82}
were separately checked using 16 threads on a
3.1GHz Intel\regsym{} Xeon\regsym{} E5-2687W.
The latter computation completed in approximately 18.5 hours.
Timings are displayed in Table~\ref{qgt1000}.

\begin{table}[ht]
\centering
\begin{tabular}{|ccccccc|}
\hline
$q$ & 1171 & 1231 & 1303 & 1321 & 1327 & 1369 \\
Time & 42 s & 173 s & 262 s & 51 s & 287 s & 74 s \\
\hline
$q$ & 1381 & 1597 & 1831 & 1849 & 2011 & 2311 \\
Time & 214 s & 235 s & 153 s & 360 s & 1546 s & 2015 s  \\
\hline
$q$ & 2671 & 2731 & 3571 & 3721 & 4111 & 4951 \\
Time & 1.62 hrs & 1.47 hrs & 1993 s & 1.82 hrs & 10.5 hrs & 1.8 hrs \\
\hline
\end{tabular}
\smallskip
\caption{Timings for checking $q > 1000$, $q$ listed in Corollary~\ref{cor82}.}
\label{qgt1000}
\end{table}

\section{Quartic Extensions}\label{square}
The preceding sections have focussed on cubic extensions of finite fields,
but Cohen~\cite{CohenCubic2} also considered quartic extensions.
In Theorem
7.2 of \cite{CohenCubic2} Cohen gave conditions on whether $q\in \Lf$. We
correct some errors in this result, and, using Lemma \ref{improvedsieve}
we prove

\begin{thm}\label{rain}
Let $q$ be a prime power, and $E_4$ be the set of $1514$ prime powers
described in the Appendix (the largest of which is $102829$).
If $q\notin E_4$ then $q\in\Lf$.
Moreover, let
$$G_L = \{2, 3, 4, 5, 7, 8, 9, 11, 13, 17, 19, 23, 25, 27, 29, 31, 37, 41, 43, 47, 73\};$$
if $q\in G_L$ then $q\notin\Lf$.
\end{thm}

We give a sketch of the proof of Theorem~\ref{rain}.
From Proposition 7.1 in \cite{CohenCubic2} we need only consider those
$q$ such that $q^{4}-1$ has at most 14 distinct prime factors.
Applying Lemma \ref{improvedsieve} with $r=0$ gives a list of 4981
values of $q$ that require further analysis.
We now apply Lemma \ref{rain} again, using the exact value of
$\delta$ for each $q$, with $r=0,1,2,3,4$.
This establishes that $q\in \Lf$ for all but the stated 1514 values of $q$.
The computations in \S \ref{fibre}
identify the 21 genuine exceptions that make up $G_L$.

We note that Theorem 7.2 in \cite{CohenCubic2} gave $q= 25943$ as the
largest possible exception, though this appears to be an error.
This value of $q$ was used by R\'{u}a \cite{Rue}, \cite{Rue2} in a related problem
concerning finite semifields.
Correspondingly, one must update Corollary 5 of
\cite{Rue} with $q = 102829$ coming from Theorem \ref{rain}.

Let $\Tn$ be the set of prime powers $q$ such that for any
$\gamma \in \Fqn$ which generates $\Fqn$ over $\Fq$ there exists
an $a\in \Fq$ with $\gamma + a$ primitive.
The determination of those prime powers in $\Tn$ is the
\textit{translate problem} for degree $n$ extensions.
It follows trivially from the definitions that $\Ln\subseteq\Tn$,
so exceptions to the translate problem can only arise from exceptions
to the line problem.

R\'{u}a's work relies not on $q$ being in $\Ln$ but on $q$ being in $\Tn$.
While it currently seems infeasible to eliminate the remaining possible
exceptions in Theorem \ref{rain}, which is concerned with $\Lf$,
we note that more progress can be made on determining membership of $\Tf$.

\begin{thm}\label{snow}
Let $q$ be a prime power, and $E_4$ be the set of $1514$ prime powers
described in the Appendix (the largest of which is $102829$).
If $q\notin E_4$ then $q\in\Tf$.
Moreover, let
$$G_T = \{3, 5, 7, 11, 13, 17, 19, 23, 25, 29, 31, 41, 43\};$$
if $q\in G_T$ then $q\notin\Tf$.
\end{thm}


By computationally verifying some values of $q$ in \S \ref{fibre}
we can improve Theorem~\ref{rain} to Theorem~\ref{sleet}.
Similarly in \S \ref{psyllium} we improve Theorem~\ref{snow} to
Theorem~\ref{avalanche}.

\subsection{Membership of $\Lf$}
\label{fibre}

We use a similar approach to the cubic case and adjust Algorithm~\ref{alg2} to Algorithm~\ref{alg4}. As we have not yet found convenient values for the inverses of
$\omega^2 + u \omega$ and $\omega^3 + t\omega^2 + u \omega$ we must compute them
each time which appears to cost an extra 10--20\%. Unfortunately the complexity of this algorithm is $O(q^6)$.

\begin{algorithm}[ht]
\DontPrintSemicolon
\AlgoDontDisplayBlockMarkers
\SetAlgoNoEnd
\SetAlgoNoLine
\SetKwProg{Proc}{Procedure}{}{}%
\SetKwFunction{checkbetagamma}{check\_beta\_inv\_gamma}%
\SetKwFunction{checkq}{check\_q}%
\SetKwFunction{rad}{rad}%
\SetKw{KwNext}{next}%
\SetKw{KwTo}{in}%
\Proc{\checkq{$q$}}{
    \textit{Construct $\Fq$, $\Fqf$, and $\omega$}\;
    $R \leftarrow$ \rad($q^4 - 1$)\;
    \For{$0 \leq k < R$}{
	$\beta \leftarrow \omega^k$\;
	\If{$\beta$ is primitive}{
	    \KwNext $k$\;
	}
	\checkbetagamma{$\beta, 1/\omega$}\;
	\For{$u$ \KwTo $\Fq$}{
	    \checkbetagamma{$\beta, 1/(\omega^2 + u \omega)$}\;
	}
	\For{$t$ \KwTo $\Fq$}{
	    \For{$u$ \KwTo $\Fq$}{
		\checkbetagamma{$\beta, 1/(\omega^3 + t\omega^2 + u \omega)$}\;
	    }
	}
    }
}
 \;
\Proc{\checkbetagamma{$\beta, \gamma^{-1}$}}{
    \For{$a$ \KwTo $\Fqs$}{
	\If{$\beta(1 + a\gamma^{-1})$ is primitive}{
	    \Return\;
	}
    }
    FAIL\;
}
\caption{Check whether $q$ is good using reduced $(\beta, \gamma^{-1})$ pairs\label{alg4}}
\end{algorithm}

\begin{thm}\label{sleet}
Define $E_L = (E_4 \cap \{x : x > 200\}) \setminus
\{239, 241, 243, 251, 257, 577\}$,
a set with $1448$ elements and largest member $102829$,
and let $q$ be a prime power not in $G_L$.
If $q\notin E_L$ then $q\in\Lf$.
\end{thm}

We give some timings for computations which check that some other possible
exceptions are not genuine exceptions in
Tables~\ref{quartic_line_table},~\ref{quartic_line_table_gt_128} and~\ref{quartic_line_table_gt_188}. Again these timings use
\MagmaV, 2.6GHz Intel\regsym{} Xeon\regsym{} E5-2670 or a similar machine.

\begin{table}[ht]
\begin{centering}
\begin{tabular}{|l|r|r|r|r|r|r|}
\hline
$q$ range & (15, 50) & (50, 100) & (100, 127] \\
\hline
Minimum & 57.82s & 1692s & 6.1 hrs \\
Average & 470.99s & 14.6 hrs & 3.212 days \\
Maximum & 805.74s & 2.72 days & 13.4 days \\
\hline
\end{tabular}
\smallskip
\caption{Timings for checking the line problem $q < 128$.}\label{quartic_line_table}
\end{centering}
\end{table}

For $q > 128$ we group our timings according to the product of
$q^2$ and the radical of $q^4-1$, as this has
substantial influence on the computation. We provide minimum, maximum and average times for these ranges.
Note that for $q > 188, q^4 > 2^{30}$
so the efficient Zech logarithm
representation cannot be used and verifying that such $q$ are not
exceptions becomes substantially more expensive. We have been able to
verify that only a few $q > 188$ are not exceptions:  these
all have minimal radical among such $q$. We separate timings for $q < 188$ and
$q > 188$ and note that in contrast to the cubic case where the general Magma
implementation could not handle $q$ with $q^3 > 2^{30}$ it can handle some $q$
with $q^4 > 2^{30}$, i.e.\ $q > 188$ as the subfield $\F_{q^2}$ can use the more
efficient representation when $q^2 < 2^{30}$, this occurs for $q < 2^{15} \sim 32000$.

We have checked the line problem for all $q < 200$ and for six $q > 200$.


\begin{table}[ht]
\begin{centering}
\begin{tabular}{|c|c|c|c|c|c|c|c|c|c|c|c|c|}
\hline
$q^2 R(q^4-1)$ (millions) & (50 000, 150 000) & (250 000, 450 000)\\
$q$ & {163, 151} & {149, 157, 137}\\
\hline
Minimum & 5.214 days (163) & 15.724 days (149)\\
Average & 7.2 days & 21.71 days\\
Maximum & 9.1 days (151) & 25.9 days (157) \\
\hline
$q^2 R(q^4-1)$ (millions) & (600 000, 1 500 000) & (3 000 000, 4 500 000) \\
$q$ & {131, 139, 167, 179, 169, 181} &  {173, 128} \\
\hline
Minimum & 38.143 days (131) & 152.8105 days (128) \\
Average & 90.27 days & 212.967 days \\
Maximum & 147.4 days (169) & 273.123 days (173) \\
\hline
\end{tabular}
\smallskip
\caption{Timings for checking the line problem for $128 \le q < 188.$}\label{quartic_line_table_gt_128}
\end{centering}
\end{table}

\begin{table}[ht]
\begin{centering}
\begin{tabular}{|c|c|c|c|c|c|c|c|c|c|c|c|c|}
\hline
$q^2 R(q^4-1)$ & (2 500, 150 000) & (250 000, 400 000) & (400 000, 600 000) \\
(millions) & & & \\
$q$ & {239, 193, 251, 257} & {199, 197, 191} & {577, 243, 241} \\
\hline
Minimum & 2 days (239) & 118.947 days (199) & 160.715 days (243) \\
Average & 21.04 days & 184.47 days& 257.142 days\\
Maximum & 59.1 days (257) & 218.67 days (197) & 315.59 days (577) \\
\hline
\end{tabular}
\smallskip
\caption{Timings for checking the line problem for some $188 < q < 600.$}\label{quartic_line_table_gt_188}
\end{centering}
\end{table}

Note that some of these timings for $q \ge 128$ are not the best possible.
We had to
split jobs into several subjobs to adhere to the 21-day limit of the machines. There are a number of $q$
for which, knowing the timings above, we could divide into subjobs more
efficiently and avoid any overlap. This would run the checks in less time. However, the returns are not
worth the extra computing resources to
rerun all these jobs for this paper.

We estimate that verifying that the remaining 1448 possible exceptions are
not genuine exceptions will take over $3.6\times10^{17}$
years using
$q^2 R(q^4-1) / (577^2 R(577^4-1)) \times 315$ (days)
for $q \in E_L, q > 260$. 


But if we use an implementation which precomputes the logs of all elements of
$\F_{q^4}$ then we have seen an improvement in timings for $q = 239, 251$
by a factor of about 3.

In Table~\ref{line_4_estimates} we give some estimates for the time it would
take to check that some $q$ are in $\Lf$ using the implementation which
precomputes the logs of all elements. Note that if surprisingly
$q\not\in\Lf$ then this can be determined in less time than estimated.
We also give the percentage of the potentially bad pairs $(\beta, \gamma^{-1})$
that we have checked of the total number of potentially bad pairs which need
to be checked, $(q^2+q+1) R(q^4-1)$.

\begin{table}[ht]
\begin{tabular}{|c|c|c|c|c|c|c|c|c|c|c|c|}
\hline
$q$ & 211 & 223 & 227 & 229 & 233 & 263 & 269 & 271 \\
\hline
$(q^4-1)/R(q^4-1)$ & 8 & 64 & 8 & 8 & 48 & 16 & 72 & 288 \\
$q^2 R(q^4-1)$ & & & & & & & & \\
(million millions) & 11 & 1 & 17 & 18 & 3 & 20 & 5 & 1 \\
\hline
Percent checked & 2.83 & 15.53 & 2 & 1.5 & 9.8 & 1 & 5.5 & 21.4 \\
Estimated full & & & & & & & & \\
check time (years) & 6.9 & 1.9 & 9.55 & 12.6 & 2.02 & 19 & 3 & 0.821 \\
\hline
$q$ & 289 & 293 & 307 & 337 & 343 & 383 & 443 & 449 \\
\hline
$(q^4-1)/R(q^4-1)$ & 192 & 280 & 600 & 416 & 240 & 256 & 1000 & 1920 \\
$q^2 R(q^4-1)$ & & & & & & & & \\
(million millions) & 3 & 2& 1& 3& 6& 12& 7& 4\\
\hline
Percent checked & 11.01 & 11.4 & 13.51 & 6.5 & 3.5 & 1.2 & 2.6 & 5 \\
Estimated full & & & & & & & & \\
check time (years) & 1.8 & 1.54 & 1.12 & 2.5 & 4.7 & 7.14 & 4.9 & 2.73 \\
\hline
\end{tabular}\smallskip\caption{Estimates for successfully checking some $q \in \Lf$.}\label{line_4_estimates}
\end{table}

\subsection{Membership of $\Tf$}
\label{psyllium}

The computation for this problem is much cheaper and a straightforward
algorithm is at worst $O(q^4)$ --- it is likely closer to $O(q^3)$ in practice
since an $a$ is usually found in a few iterations. We first tried iterating
through all $\beta \in \F_{q^4}$ but this is at worst $O(q^5)$ and at best
$O(q^4)$. We found it best to iterate through all
$\beta = \lambda_3 \omega^3 + \lambda_2 \omega^2 + \lambda_1 \omega$ of which
there are $O(q^3)$. For those $\beta \notin \F_{q^2}$ we checked that there is
an $a$ such that $\beta + a$ is primitive.

We only need to check those $q$ which are genuine exceptions to the line problem
and those $q$ which were too expensive to check for the line problem.
It took 7.5s in total
(using \MagmaV, 2.6GHz Intel\regsym{} Xeon\regsym{} E5-2670)
to check that $q = 2, 4, 8, 9, 27, 37, 47, 73$ are not genuine exceptions to the translate
problem.

\begin{thm}\label{avalanche}
Define $E_T = E_4 \cap (\{x : \ x > 23000 \}$ 
a set with $124$ elements and largest member $102829$,
and let $q$ be a prime power not in $G_T$.
If $q\notin E_T$ then $q\in\Tf$.
\end{thm}
We remark that all even prime powers $2^e$ are in $\Tf$.
This improves Corollary 2 in \cite{Rue2}: one may remove `$\Tf \cap$' from this Corollary.

We give some timings in Tables~\ref{quartic_trans_table_lt_2000} and~\ref{quartic_trans_table_gt_3500}
for computations which check that some other possible exceptions $q > 200$ are not
genuine exceptions.
Again these timings use
\MagmaV, 2.6GHz Intel\regsym{} Xeon\regsym{} E5-2670 or a similar machine.
\begin{table}[ht]
\begin{centering}
\begin{tabular}{|l|c|c|c|c|c|c|}
\hline
$q$ range & (200, 500) & (500, 750) & (750, 1000) \\
\hline
Minimum & 436.92s (256) & 2.2 hrs (509) & 7.5 hrs (773) \\
Average & 0.87 hrs & 4.6222 hrs & 13.3 hrs \\
Maximum & 2.7 hrs (463) & 10.2 hrs (727) & 25.2 hrs (967) \\
\hline
$q$ range & (1000, 1250) & (1250, 1500) & (1500, 1750) \\
\hline
Minimum & 17.6 hrs (1039) & 1.5 days (1283) & 2.86 days (1531) \\
Average & 29.33 hrs & 2.43 days & 4.394 days \\
Maximum & 2 days (1217) & 4.415 days (1483) & 6.81 days (1747) \\
\hline
$q$ range & (1750, 2000)& (2000, 2250) & (2250, 2500)  \\
\hline
Minimum & 4.6 days (1811)& 6.84 days (2039) & 11.3 days (2281)  \\
Average  & 7 days & 14.543 days & 19.5 days \\
Maximum & 10.41 days (1973) & 24.13 days (2207) & 27.156 days (2393)  \\
\hline
$q$ range & (2500, 2750) & (2750, 3000)& (3000, 3250) \\
\hline
Minimum & 16.325 days (2539) & 23.84 days (2797)& 31.6 days (3019)  \\
Average & 25.982 days & 34.33 days & 46.009 days \\
Maximum & 43.101012 days (2729) & 49.924 days (2969) & 69.326 days (3191) \\
\hline
$q$ range & (3250, 3500) & (3500, 3700) &  \\
\hline
Minimum & 42.9233 days (3391) & 57.245 days (3517) & \\
Average & 60.033 days & $\ge$ 70 days & \\
Maximum & 85 days (3433) &$\ge$ 90 days & \\
\hline
\end{tabular}\smallskip\caption{Timings for checking the translate problem using a general implementation.}\label{quartic_trans_table_lt_2000}
\end{centering}
\end{table}
We estimate that checking the remaining elements of $E_T$ 
will take 
30000 years. We calculated this using the timing for $q = 3019$ which
was minimal in its range.
For $q \in E_T, q > 4000,  \mathrm{min}_q\{(q/3019)^3\times 31.6\} \sim 73$
and $\mathrm{max}_q\{(q/3019)^3\times 31.6\} \sim 3420 \times 365$
so that the time taken to check these
$q$ will be more than 73 days each and there will be a $q$ which will take at least
3420 years to check. The average estimate for checking these $q$ is 32 years.

Looking at the more achievable, checking all $4000 < q < 5000$ may take 30
years, or 1 year using 30 processors efficiently. Each $q < 6825$ may be
able to be checked in at most about 1 year each although there are 267 such $q$,
163 more than $4000 < q < 5000$.


A specialised implementation to precompute the logs of elements of
$\F_{q^4}$  reduces overhead.  We observed an improvement by a factor of
about 24 for $q \in [3500, 4000]$, that is, computations which took about $x$ days in the
general implementation take about $x$ hours in the specialised implementation.
The timings in Table~\ref{quartic_trans_table_gt_3500} are a start on the
use of this implementation. Most of these timings will not be best possible. 
The maximum number of threads was used simultaneously rather than the number of
cores of the machine. 
While this increases the timings, this choice was made in order to check 
as many $q$ as possible in the shortest amount of real time (twice as many
computations could be done on 2 threads per core in less than twice the time).

\begin{table}[ht]
\begin{centering}
\begin{tabular}{|l|c|c|c|c|c|c|}
\hline
$q$ range & (3500, 3750) & (3750, 4000) & (4000, 4250) \\
\hline
Minimum & 2.051 days (3517) & 2.724 days (3881) & 30.6 hrs (4096) \\
Average & 2.6 days & 3.26 days & 3.8 days \\
Maximum & 3.77 days (3739) & 3.8 days (3947) & 4.87 days (4159) \\
\hline
$q$ range & (4250, 4500) & (4500, 4750) & (4750, 5000) \\
\hline
Minimum & 2.93 days (4489) & 4.64 days (4547) & 4.14 days (4913) \\
Average & 4.7 days & 5.7 days & 6.604 days \\
Maximum & 5.589 days (4271) & 6.734 days (4523) & 7.5 days (4831) \\
\hline
$q$ range  & (5000, 5250) & (5250, 5500) & (5500, 5750) \\
\hline
Minimum & 6.4 days (5003) & 4.63 days (5329) & 8.93 days (5519) \\
Average & 7.8 days & 8.9 days & 10.6 days \\
Maximum & 9.8 days (5237) & 10.332 days (5347) & 16.395 days (5741) \\
\hline
$q$ range & (5750, 6000) & (6000, 6500) & (6500, 7000) \\
\hline
Minimum & 10.34 days (5801) & 7.36 days (6241) & 11.351 days (6859) \\
Average & 12.414 days & 14.7 days & 18.2 days \\
Maximum & 15.925 days (5851) & 20.61 days (6469) & 23.37 days (6733) \\
\hline
$q$ range & (7000, 7500) & (7500, 8000) & (8000, 8750) \\
\hline
Minimum & 20.1244 days (7001) & 16.8 days (7921) & 27.2 days (8011) \\
Average & 23.4 days & 28 days & 36.9 days \\
Maximum & 28.8 days (7253) & 36.4 days (7853) & 47.1 days (8581) \\
\hline
$q$ range  & (8750, 9500) & (9500, 10500) & (10500, 11500) \\
\hline
Minimum & 37.81 days (8783) & 50.24 days (9511) & 70.6 days (10501) \\
Average & 49.1 days & 64.6 days & 89.424 days \\
Maximum & 61.6 days (9437) & 79.2 days (10267) & 107.44 (11467)\\
\hline
$q$ range & (11500, 12500) & (12500, 13500) & (13500, 15000) \\
\hline
Minimum & 61 days (11881) & 108.54 days (12919) & 172.3 days (13697) \\
Average & 110.4 days & 152.4 days & 215.2 days \\
Maximum & 139.2 days (12277) & 191 days (13469) & 256 days (14939) \\
\hline
$q$ range & (15000, 17000) & (17000, 19500) & (19500, 23000) \\
\hline
Minimum & 212.2 days (15053) & 345.4 days (17573) & 625 days (20201) \\
Average & 296.65 days & 535 days & 823 days \\
Maximum & 405.3 days (16927) & 1190 days (19181) & 1234 days (21319) \\
\hline
\end{tabular}\smallskip\caption{Timings for checking the translate problem using a specialised implementation.}\label{quartic_trans_table_gt_3500}
\end{centering}
\end{table}

We therefore arrive at $E_{T}$, the list of possible exceptions in Theorem~\ref{avalanche}. We note that the size of $E_{T}$ is 8.2\% of the size of $E_4$.

We estimate that, using this specialised implementation, checking $q = 25037$
will take around 3.5 years and checking $q = 102829$ will take 457 years,
for which a considerable amount of resources would be needed.

\section*{Acknowledgements}
The authors would like to thank Allan Steel for the change of
a {\sc Magma} setting with regard to the representation used for finite fields. The
third author would also like to thank Allan Steel for help in accessing the
HPC resources at the University of Sydney.
We acknowledge the Sydney Informatics Hub and the University of Sydney's high
performance computing cluster Artemis for providing the
resources that have contributed to results in  Section~\ref{square} of this paper.

\section*{Appendix}

We describe here the set $E_{4}$ of possibly exceptional $q$ for the
quartic extension problems.
We start with a list of prime powers up to 9620, then exclude 198
values which are not exceptions, then add in a further 474 larger potential
exceptions.

$E_4 = (\{q : 2 \leq q \leq 9620, \; q = p^{\alpha}\}\setminus\{$
2048, 2187, 3491, 3701, 3721, 3803, 3833, 3889, 3967, 4021, 4057, 4079, 4099,
4177, 4253, 4349, 4457, 4561, 4567, 4639, 4651, 4703, 4721, 4723, 4799, 4801,
4933, 5009, 5021, 5041, 5051, 5077, 5119, 5233, 5297, 5399, 5437, 5441, 5443,
5449, 5471, 5483, 5527, 5639, 5651, 5717, 5791, 5879, 5987, 6011, 6047, 6101,
6113, 6121, 6143, 6197, 6199, 6211, 6317, 6361, 6367, 6373, 6389, 6529, 6547,
6561, 6563, 6619, 6653, 6659, 6673, 6701, 6737, 6781, 6793, 6823, 6829, 6857,
6871, 6883, 6899, 6907, 6911, 6949, 6961, 7027, 7057, 7109, 7121, 7159, 7211,
7213, 7219, 7247, 7351, 7417, 7451, 7499, 7507, 7529, 7537, 7541, 7559, 7573,
7577, 7607, 7681, 7691, 7703, 7723, 7757, 7759, 7793, 7817, 7823, 7829, 7901,
7907, 7927, 7933, 7949, 7993, 8053, 8069, 8081, 8087, 8089, 8101, 8111, 8123,
8167, 8192, 8209, 8221, 8231, 8263, 8269, 8287, 8291, 8311, 8353, 8369, 8389,
8423, 8431, 8447, 8461, 8521, 8543, 8563, 8573, 8599, 8629, 8641, 8677, 8699,
8713, 8719, 8747, 8753, 8761, 8803, 8831, 8837, 8893, 8941, 8951, 8963, 9001,
9013, 9041, 9049, 9067, 9091, 9103, 9109, 9137, 9151, 9161, 9187, 9209, 9241,
9277, 9293, 9319, 9341, 9343, 9377, 9391, 9403, 9409, 9419, 9467, 9473, 9497,
9539, 9551, 9601
$\}) \ \cup$\\
$\{$
9661, 9677, 9689, 9749, 9767, 9781, 9787, 9803, 9811, 9829, 9833, 9857, 9859,
9871, 9901, 9907, 9941, 9967, 10009, 10037, 10061, 10067, 10093, 10141, 10163,
10169, 10177, 10193, 10223, 10247, 10259, 10267, 10301, 10303, 10331, 10427,
10429, 10457, 10459, 10477, 10487, 10499, 10501, 10597, 10613, 10627, 10639,
10709, 10711, 10723, 10739, 10781, 10789, 10837, 10847, 10859, 10867, 10889,
10949, 10973, 10979, 11003, 11059, 11071, 11087, 11117, 11119, 11131, 11159,
11173, 11177, 11213, 11243, 11257, 11287, 11311, 11351, 11353, 11369, 11383,
11411, 11423, 11437, 11467, 11471, 11527, 11549, 11551, 11579, 11593, 11617,
11621, 11717, 11731, 11743, 11777, 11779, 11783, 11789, 11831, 11867, 11881,
11887, 11903, 11927, 11933, 11969, 11971, 11981, 12007, 12011, 12143, 12167,
12211, 12227, 12241, 12253, 12277, 12323, 12329, 12377, 12391, 12401, 12409,
12433, 12473, 12503, 12511, 12517, 12583, 12611, 12613, 12637, 12641, 12671,
12689, 12697, 12713, 12739, 12743, 12781, 12823, 12893, 12907, 12919, 12923,
12953, 12959, 12979, 13001, 13033, 13049, 13099, 13103, 13109, 13159, 13187,
13259, 13267, 13313, 13331, 13339, 13397, 13399, 13417, 13441, 13451, 13463,
13469, 13553, 13567, 13597, 13613, 13697, 13723, 13757, 13763, 13831, 13859,
13883, 13903, 13931, 14029, 14057, 14071, 14153, 14197, 14251, 14281, 14321,
14323, 14327, 14431, 14449, 14461, 14519, 14533, 14629, 14633, 14669, 14741,
14783, 14827, 14851, 14867, 14897, 14923, 14939, 14947, 14951, 15053, 15107,
15131, 15137, 15287, 15289, 15313, 15329, 15391, 15401, 15443, 15497, 15511,
15527, 15541, 15569, 15581, 15619, 15641, 15731, 15809, 15817, 15907, 15959,
16073, 16103, 16141, 16183, 16253, 16301, 16339, 16421, 16453, 16477, 16567,
16619, 16633, 16661, 16759, 16763, 16787, 16829, 16843, 16883, 16927, 17029,
17093, 17137, 17191, 17203, 17207, 17291, 17341, 17359, 17387, 17389, 17401,
17467, 17573, 17579, 17597, 17681, 17837, 17863, 17909, 17939, 18041, 18061,
18089, 18127, 18143, 18257, 18311, 18353, 18427, 18481, 18493, 18517, 18679,
18773, 18787, 18803, 18869, 18899, 19139, 19141, 19163, 19181, 19183, 19319,
19381, 19391, 19417, 19447, 19469, 19531, 19571, 19597, 19609, 19739, 19753,
19843, 19937, 19963, 19993, 20021, 20047, 20129, 20201, 20327, 20399, 20483,
20549, 20593, 20707, 20747, 20749, 20899, 21013, 21169, 21319, 21407, 21419,
21433, 21517, 21559, 21713, 21727, 21757, 21803, 21841, 21943, 22079, 22133,
22147, 22303, 22469, 22511, 22541, 22877, 23057, 23087, 23143, 23269, 23297,
23311, 23321, 23473, 23549, 23561, 23563, 23827, 23869, 23971, 23981, 24023,
24179, 24389, 24509, 24611, 24683, 24851, 24907, 25037, 25117, 25423, 25453,
25537, 25577, 25943, 25997, 26083, 26417, 26489, 26573, 26597, 26839, 26893,
27061, 27763, 28183, 28309, 28573, 28643, 28657, 29147, 29173, 29303, 29347,
29567, 29611, 29717, 30103, 30211, 30269, 30493, 30689, 30757, 31123, 31151,
31247, 31667, 32117, 32297, 32369, 32381, 32423, 32537, 32843, 32869, 33797,
34033, 34429, 34693, 35531, 35771, 36037, 36583, 36653, 36821, 36847, 37253,
37549, 37591, 38011, 38039, 38303, 38501, 38611, 38917, 39733, 39929, 40039,
40699, 41117, 41777, 41887, 42223, 42589, 43889, 44507, 46619, 46663, 48313,
49477, 50051, 50653, 52571, 53087, 53129, 53591, 53923, 54319, 55021, 56393,
57793, 58787, 59093, 59753, 60397, 63601, 66347, 73039, 102829
$\}$

\end{document}